\documentclass[10.5pt]{amsart}

\usepackage{amsfonts}
\usepackage{amssymb}
\usepackage{setspace}
\usepackage{float}
\usepackage{cite}

\usepackage{geometry}
\usepackage[utf8]{inputenc}
\usepackage[english]{babel}
\usepackage{enumitem}

\usepackage{graphicx}
 
\usepackage{amsthm}
\usepackage{amsmath}
\usepackage{amssymb}

\newtheorem{theorem}{Theorem}

\newtheorem{lemma}{Lemma}
\newtheorem{prop}[theorem]{Proposition}

\theoremstyle{remark}
\newtheorem{exmp}{Example}

\renewcommand{\leq}{\leqslant}
\renewcommand{\geq}{\geqslant}

\title{Residue races of the number of prime divisors function}

\author[S. Porritt]{Sam Porritt}
\address{Department of Mathematics\\University College London\\
25 Gordon Street, London, England}
\email{samuel.porritt.15@ucl.ac.uk}

\begin{document}

\onehalfspacing

\maketitle

\begin{abstract}
We investigate the distribution of the function $\omega(n)$, the number of distinct prime divisors of $n$, in residue classes modulo $q$ for natural numbers $q$ greater than 2. In particular we ask `prime number races' style questions, as suggested by Coons and Dahmen in their paper `On the residue class distribution of the number of prime divisors of an integer'.
\end{abstract}

\section{Introduction}

Let $q \geq 2$ be an integer, $a\in\{0,1,\ldots, q-1\}$ represent some residue class modulo $q$ and $\omega(n)$ denote the number of distinct prime divisors of $n$. Define
$$N_{a,q}(x):=\#\{n \leq x : \omega(n)\equiv a \mod q\}.$$
Seeing no reason why $\omega(n)$ should favour any particular residue class, we expect that for all $a$,
\begin{equation}
N_{a,q}(x)\sim \frac{x}{q} \:\:\:\: \text{ as } \:\:\: x \rightarrow \infty.
\end{equation}

In fact, it was proved in~\cite{Add} that 
$$N_{a,q}(x) - \frac{x}{q} = O\left(\frac{x}{(\log x)^{c(q)}}\right)$$
with $c(q) = 1-\cos(\frac{2\pi}{q})$. It was also proved that for $q>2$ the error term here is best possible, since it was also determined that for $q>2$
$$N_{a,q}(x) - \frac{x}{q} = \Omega_{\pm}\left(\frac{x}{(\log x)^{c(q)}}\right).$$ This is in stark contrast to the case $q=2$ for which we expect ``square-root cancellation". Indeed,
$$
N_{a,2}(x) - \frac{x}{2} =O(x^{1/2+o(1)}) \:\: \text{  for  }  a=0 \text{  and  } 1 
$$
is equivalent to the Riemann Hypothesis. For $q=2$, it is well known that (1) is equivalent to the prime number theorem. 

In~\cite{Coons-Dah}, the authors suggest that, in the spirit of prime number races, it would be interesting to investigate the sign changes of $N_{a,q}(x)-N_{b,q}(x).$ The traditional prime number races concern the popularity of residue classes for prime numbers rather than for the values of $\omega$, that is, sign changes of $\pi(x;a,q)-\pi(x;b,q)$ where $\pi(x;a,q)$ is the number of primes less than or equal to $x$ which are congruent to $a$ modulo $q$. Rubinstein and Sarnak~\cite{Rub-Sar} proved under certain reasonable assumptions that the set
$$\{x \in \mathbb{N} \: : \:\pi(x;3,4)<\pi(x;1,4)\},$$
for example, does not have a natural density in the integers but does have a \textit{logarithmic density}, defined for a subset $E\subset \mathbb{N}$, if the limit exists, to be $\lim_{X\rightarrow \infty}\frac{1}{\log X}\sum_{\substack{ x\leq X \\ x \in E }}\frac{1}{x}$. Loosely speaking, the reason for this is that the difference $\pi(x;3,4) - \pi(x;1,4)$ can be written  as a sum of terms of the form $\sin(\gamma \log x)/\gamma$, where $\gamma$ ranges over the imaginary parts of the zeros of certain Dirichlet $L$-functions. For an introduction to this topic see~\cite{Gran-Martin}.

For our investigations it is natural to consider mean values of the multiplicative functions $n\mapsto z^{\omega(n)}$ where $z$ is taken to be a complex $q$-th root of unity. By applying a classical result first due to Selberg concerning such mean values we will establish an asymptotic formula for $N_{a,q}(x)$ with main term and next highest order term in the case $q>2$.  This will be used to prove our main theorem. The formula will contain an expression of the form $\cos(A \log \log x + B)$ and so in our case we have neither natural nor logarithmic density, but instead need to go further and define the notion of \emph{loglog density}. We say a subset $E \subset \mathbb{N}$ has \emph{loglog density} $\delta$ if $$\frac{1}{\log \log X}\sum_{\substack{ x\leq X \\ x \in E }}\frac{1}{x\log x} \rightarrow \delta \text{  as  } X \rightarrow \infty.$$ With this we can now state our main theorems.

\begin{theorem}\label{thm1}
Let $q>2$ be an integer and $a,b\in \{0,1,\ldots, q-1\}$ with $a\neq b$. The set $$E_{a,b}:=\{x \in \mathbb{N} \: :\: N_{a,q}(x) < N_{b,q}(x) \}$$ has no natural density, in fact $$\liminf_{X\rightarrow \infty}\frac{1}{X}[1,X]\cap E_{a,b} = 0 \:\:\text{  and  }\:\:  \limsup_{X\rightarrow \infty}\frac{1}{X}[1,X]\cap E_{a,b} = 1.$$
\end{theorem}

\begin{theorem}\label{thm2}
The set $E_{a,b}$, defined in Theorem~\ref{thm1}, has no natural or logarithmic density, but has \emph{loglog density} equal to $1/2$.
\end{theorem}

Given a complete ordering on the residue classes, we can also ask how often the different `competitors' in our race are in that order.

\begin{theorem}\label{thm3}
Let $q>2$ be an integer and $a \in \{0, 1, \ldots, q-1\}$. Each of the following sets has \emph{loglog density} $\frac{1}{2q}$
$$U_{a,q}:=\{x \in \mathbb{N} \: :\: N_{a, q}(x)>N_{a-1, q}(x)>N_{a+1, q}(x)>\cdots>N_{a-i,q}(x)>N_{a+i,q}(x)>\cdots \}$$
$$V_{a, q}:=\{x \in \mathbb{N} \: :\: N_{a, q}(x)>N_{a+1, q}(x)>N_{a-1, q}(x)>\cdots>N_{a+i,q}(x)>N_{a-i,q}(x)>\cdots \}.$$ Therefore, since there are $2q$ of them, these are the only permutations which appear with non-zero \emph{loglog densities.}
\end{theorem}

\begin{exmp}
When $q=6$ and $a=4$ we get
$$\lim_{X\rightarrow \infty}\frac{1}{\log \log X}\sum_{\substack{x \leq X \\ N_{4, 6}(x)>N_{3, 6}(x)>N_{5, 6}(x)>N_{2, 6}(x)>N_{0, 6}(x)>N_{1, 6}(x)}}\frac{1}{x\log x}=\frac{1}{12}$$
$$\lim_{X\rightarrow \infty}\frac{1}{\log \log X}\sum_{\substack{x \leq X \\ N_{4, 6}(x)>N_{5, 6}(x)>N_{3, 6}(x)>N_{0, 6}(x)>N_{2, 6}(x)>N_{1, 6}(x)}}\frac{1}{x\log x}=\frac{1}{12}.$$
\end{exmp}

Theorem~\ref{thm1} follows easily from the proofs of Proposition~\ref{prop7} and Theorem~\ref{thm3}. Theorem~\ref{thm2} follows from Theorem~\ref{thm3} because, of the $2q$ permutations with non-zero loglog density, there are $1/q$ in which $N_{a,q}(x)<N_{b,q}(x)$.

We also prove that certain orderings can occur only a finite number of times.

\begin{theorem}\label{thm4}
Any permutation $\sigma$ of $\{0,1,\ldots, q-1\}$ for which the set $$\{x\in \mathbb{N}\: :\: N_{\sigma(0),q}(x)\geq N_{\sigma(1), q}(x) \geq \ldots \geq N_{\sigma (q-1),q}(x) \}$$ is infinite is such that $\sigma(0) = \sigma (1) \pm 1 \mod q$.
\end{theorem}

Notice that there are $2q!/(q-1)$ such permutations so if $q$ is large, a vanishingly small proportion of the possible permutations occur infinitely often.

\begin{exmp}
If $q=4$ then only 16 out of the 24 orderings can occur for arbitrarily large $x$. There is a point after which, if ``0 is in the lead", then 2 cannot be second and vice versa. Similarly, they cannot simultaneously hold positions 3rd and 4th, and the same goes for the pair of residue classes 1 and 3.
\end{exmp}

Let us look at the start of the mod 4 race before moving on to the proofs. For a better view, the mean has been subtracted and the points are plotted on a loglog scale.

\begin{figure}[H]
{
\centering
\includegraphics[scale=0.45]{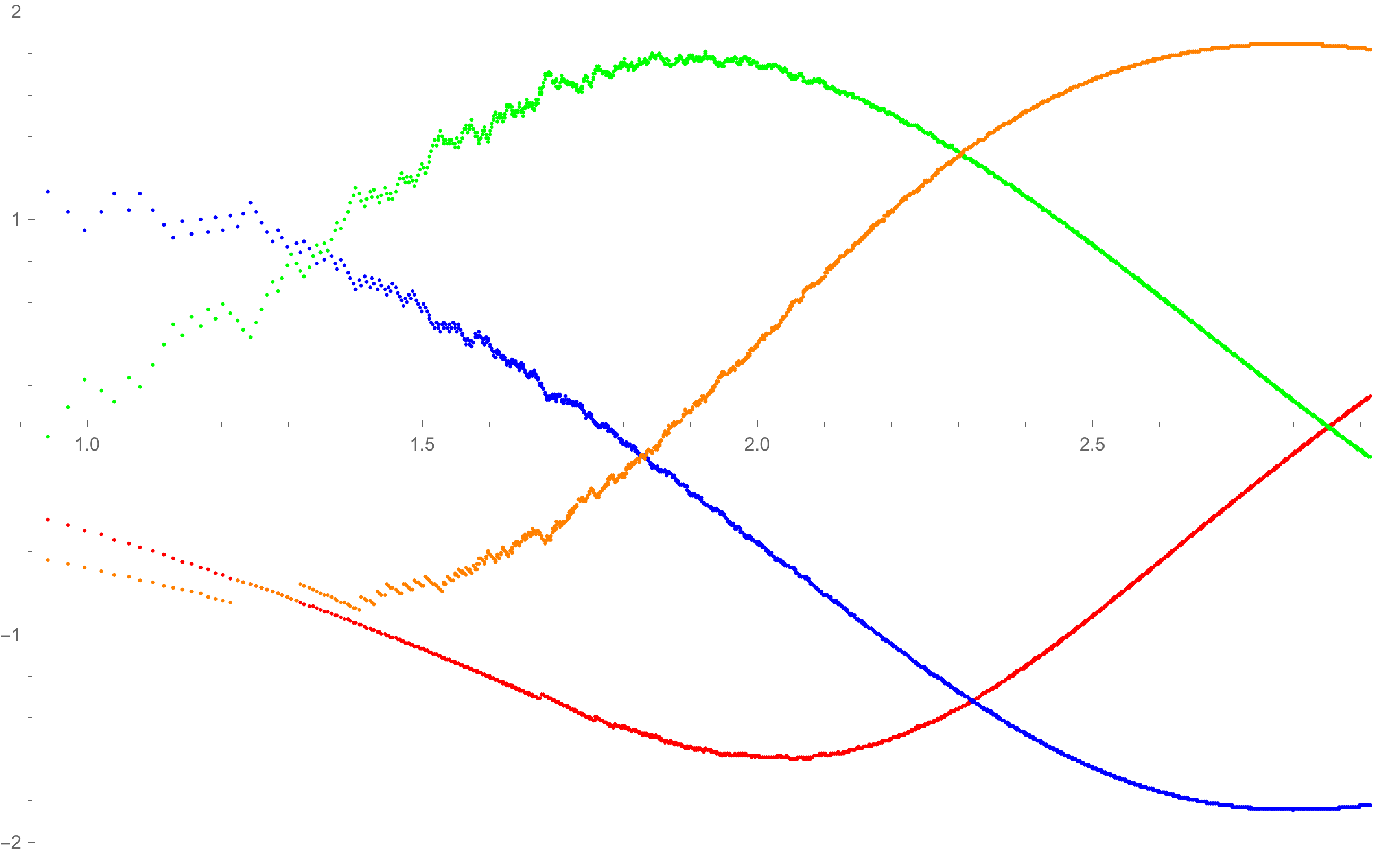}
\caption{$\omega$ mod 4 race}
}
\end{figure}

The plotted points along with their colours are as follows:
\begin{align*}
&\text{ Red: }\{\;\big(\log\log x, (\frac{1}{x}N_{0,4}(x)-\frac{1}{4})\log x\big) \::\:10\leq x \leq 10^8\},\\
&\text{ Blue: }\{\;\big(\log\log x, (\frac{1}{x}N_{1,4}(x)-\frac{1}{4})\log x\big) \::\:10\leq x \leq 10^8\},\\
&\text{ Orange: }\{\;\big(\log\log x, (\frac{1}{x}N_{2,4}(x)-\frac{1}{4})\log x\big) \::\:10\leq x \leq 10^8\},\\
&\text{ Green: }\{\;\big(\log\log x, (\frac{1}{x}N_{3,4}(x)-\frac{1}{4})\log x\big) \::\:10\leq x \leq 10^8\}.
\end{align*}

This data strongly suggests, at least for the $q=4$ race, that any ordering not of the form stated in Theorem~\ref{thm4} can $\emph{never}$ occur. It may not be unreasonable to conjecture that this is the case for all $q$.

We remark that similar results can be proved for the sets $$\{n\leq x \: : \: \Omega(n) \equiv a \mod q\} \text{   and   } \{n, \text{ square-free }\leq x \: : \: \omega(n) \equiv a \mod q\}$$ where $\Omega(n) = \sum_{p^k|n}1$ counts prime divisors with multiplicities.

\section{Preliminaries}

To save space, we will use $\log_2 x$ and $e_2(x)$ to denote $\log\log x$ and $\exp(\exp(x))$ respectively. We start by proving an asymptotic formula for $N_{a,q}(x)$.

\begin{prop}\label{prop5}
For $x\geq 3$ we have $$N_{a,q}(x) = \frac{x}{q}\left(1+2|g(\phi)|\cos\Big(\sin\Big(\frac{2\pi}{q}\Big)\log_2 x + \theta -\frac{2\pi a}{q}\Big)(\log x)^{\cos(\frac{2\pi}{q})-1} + O((\log x)^{\cos(\frac{4\pi}{q})-1})\right)$$

where $g(z):=\frac{1}{\Gamma(z)}\prod_p \Big(1+\frac{z}{p-1}\Big)(1-p^{-1})^z$ and $\phi = \phi(q) = e^{2\pi i /q}$ and $\theta = \arg g(\phi)$.
\end{prop}

\begin{proof}

The main result we will make use of is [1, Theorem 2] from which it follows that for $z$ a complex variable bounded in absolute value by 1, and $g(z)$ defined as in the proposition we have
\begin{equation}A_z(x):=\sum_{n\leq x}z^{\omega(n)}=g(z)x(\log x)^{z-1} + O\big(x(\log x)^{\Re z -2})\big).\end{equation}
These mean values satisfy
$$A_{\phi^j}(x)=\sum_{n\leq x}\phi^{j\omega(n)}=\sum_{k=0}^{q-1}\phi^{jk}N_{k,q}(x)\:\:\text{  for all } j = 0, \ldots, q-1.$$
We can isolate the $N_{a,q}(x)$ as follows
$$\frac{1}{q} \sum_{j = 0}^{q-1}\overline{\phi}^{aj} A_{\phi^j} = \frac{1}{q} \sum_{j=0}^{q-1} \overline{\phi}^{aj} \sum_{k=0}^{q-1} \phi^{jk} N_{k,q}(x)
= \sum_{k=0}^{q-1}  N_{k,q}(x)  \frac{1}{q} \sum_{j=0}^{q-1}  \phi^{j(k-a)} = N_{a,q}(x).$$
Substituting (2) into the line above gives
$$N_{a,q}(x)=\frac{x}{q}\left(1+\sum_{j=1}^{q-1}\Big(\overline{\phi}^{aj}g(\phi^j)e^{i\sin(\frac{2\pi j}{q})\log_2 x}(\log x)^{\cos (\frac{2\pi j}{q})-1} + O((\log x)^{\cos (\frac{2\pi j}{q})-2})\Big) \right).$$
For sufficiently large $x$, the terms in this sum with largest absolute value are those with $j=1$ and $j=q-1$. Each of the others is $\ll (\log x)^{\cos(\frac{4\pi}{q})-1}.$ Combining this observation with the fact that $g(\overline{z})=\overline{g(z)}$ we get

\begin{align*}
N_{a,q}(x)&=\frac{x}{q}\left(1+2\Re( \overline{\phi}^a g(\phi)e^{i \sin(\frac{2\pi}{q})\log_2 x})(\log x)^{\cos(\frac{2\pi}{q})-1} + O((\log x)^{\cos(\frac{4\pi}{q})-1})\right) \\
&=\frac{x}{q}\left(1+2|g(\phi)|\cos\Big(\sin\Big(\frac{2\pi}{q}\Big)\log_2 x + \theta -\frac{2\pi a}{q}\Big)(\log x)^{\cos(\frac{2\pi}{q})-1} + O((\log x)^{\cos(\frac{4\pi}{q})-1})\right)
\end{align*}
where $\theta = \arg g(\phi).$

\end{proof}

Notice that if $q=2$ then $\phi=-1$ and $g(-1)=0$. It is for this reason we cannot say any more in this most interesting case. Indeed we actually suspect a much stronger error term of $O(x^{1/2+o(1)})$ in this case. For $q>2$ though, $g(\phi)\neq 0$, as $\Gamma$ has no pole there and the product has only non-zero terms.

From this, we see immediately that
$$N_{a,q}(x) - \frac{x}{q} = O\left(\frac{x}{(\log x)^{1-\cos(\frac{2\pi}{q})}}\right)$$
and also
$$N_{a,q}(x) - \frac{x}{q} = \Omega_{\pm}\left(\frac{x}{(\log x)^{1-\cos(\frac{2\pi}{q})}}\right).$$

Before trying to understand when $N_{a,q}(x)$ is less than this average value of $x/q$, we start with the simpler, but related, question of when the secondary term is negative. That is, when $$\cos\Big(\sin\Big(\frac{2\pi}{q}\Big)\log_2 x + \theta -\frac{2\pi a}{q}\Big)<0.$$ Now cosine is negative ``about half the time" which might suggest that $N_{a,q}(x)$ is less than its average value ``about half the time" too. To be precise, the set $\{n\in \mathbb{N}:\cos n <0\}$ has \textit{natural density} 1/2 in $\mathbb{N}$. That is, $\lim_{x\rightarrow \infty}\frac{1}{x}\sum_{\substack{ n\leq x \\ \cos n<0}}1 =\frac{1}{2}.$ It is not true though that $\{n\in \mathbb{N} : \cos (\log_2 n) < 0\}$ has natural density 1/2. In fact this set has no natural density, as we shall see below. The presence of the $\cos(\log_2 x)$ type term is why we ought to be looking at the $\log\log$ density.

The property of possessing a natural density is stronger than that of possessing a logarithmic density, which is stronger still than having a loglog density. In fact, a straightforward application of partial summation proves that if a set $E\subset \mathbb{N}$ has a natural density then it also has a logarithmic density and the two are equal, and if $E$ has a logarithmic density then it has a loglog density and the two are equal. The following lemma, which we do not prove here, reassures us that our notions of logarithmic and loglog density are sound.

\begin{lemma}\label{lem} There exist constants $\gamma$, $\mu$ such that
\begin{align*}
\sum_{1\leq n \leq x}\frac{1}{n} &= \log x + \gamma + O\Big(\frac{1}{x}\Big).\\
\sum_{2\leq n \leq x}\frac{1}{n\log n} &= \log_2 x + \mu + O\Big(\frac{1}{x \log x}\Big).
\end{align*}
\end{lemma}

\begin{prop}\label{prop7}
Let $A, B \in \mathbb{R}$ with $A > 0$. The set $$E:=\{n\in \mathbb{N} : \cos(A\log_2(n) + B)<0\}$$ has \emph{loglog density} 1/2 but no natural or logarithmic densities.
\end{prop}

\begin{proof}
For $N\in \mathbb{N}$, let $x_N:= e_2((2N\pi-\pi/2-B)/A) $ and $y_N:=e_2((2N\pi+\pi/2-B)/A)$ so that
\begin{align*}
\cos(A\log_2 x+B) > 0 &\Leftrightarrow x\in (x_N,y_N) \text{ for some } N \text{ and }\\
\cos(A\log_2 x+B) < 0 &\Leftrightarrow x\in (y_{N-1},x_{N}) \text{ for some }N.
\end{align*}
Writing $[\alpha]$ for the largest integer at most $\alpha$, we therefore have
$$\frac{1}{x_N}\sum_{\substack{n\leq x_N \\ n\in E}}1 \geq \frac{[x_N]-[y_{N-1}]}{x_N} \rightarrow 1 \text{ as } N\rightarrow \infty$$
and
$$\frac{1}{y_N}\sum_{\substack{n\leq y_N \\ n\in E}}1 \leq 1 -  \frac{[y_N]-[x_N]}{y_N} \rightarrow 0 \text{ as } N\rightarrow \infty,$$
so $E$ certainly doesn't have a natural density. When we look at the logarithmic density, we get, since $\log y_N = e^{\pi/A}\log x_N,$
$$\frac{1}{\log x_N}\sum_{\substack{n\leq x_N \\ n\in E}}\frac{1}{n} = \frac{e^{\pi/A}}{\log y_N}\sum_{\substack{n\leq y_N \\ n\in E}}\frac{1}{n},$$
so for the limit to exist as $N\rightarrow \infty$ we would need $e^{\pi/A} = 1$ which is impossible.

When we look at the loglog density however, we get, using Lemma~\ref{lem}
\begin{align*}
\limsup_{x\rightarrow \infty}\frac{1}{\log_2 x}\sum_{\substack{n\leq x \\ n\in E}}\frac{1}{n\log n} &= \lim_{N\rightarrow \infty}\frac{1}{\log_2 x_N}\sum_{m=2}^{N}\sum_{n\in(y_{m-1}, x_m)}\frac{1}{n\log n} \\
&= \lim_{N\rightarrow \infty}\frac{1}{\log_2 x_N}\sum_{m=2}^{N}\left(\log_2(x_{m})-\log_2(y_{m-1})+\mu - \mu +O\Big(\frac{1}{m\log m}\Big) \right)\\ 
&= \lim_{N\rightarrow \infty}\frac{1}{(2N\pi-\pi/2-B)/A}\sum_{m=2}^{N}\left(\pi/A + O\Big(\frac{1}{m\log m}\Big) \right)\\
&= \frac{1}{2}.
\end{align*}
A similar calculation shows that $\liminf_{x\rightarrow \infty}\frac{1}{\log_2 x}\sum_{\substack{n\leq x \\ n\log n\in E}}\frac{1}{n\log n} =\frac{1}{2}$ and the result follows.
\end{proof}

It is tempting to conclude that the set $$\{x \in \mathbb{N} \::\: N_{a, q}(x)<\frac{x}{q}\}$$ has no natural or logarithmic densities but has loglog density 1/2. Unfortunately, to prove this rigorously we will need to account for the error introduced by the terms we have left out. We will do this shortly. If we forget about error terms for the moment though (which we can only really do when $\sin(2\pi/q)\log_2 x +\theta-\frac{2\pi a}{q}$ is not too close to a zero of $\cos$), then asking ``for which value of $a$ is $N_{a,q}(x)$ largest" is tantamount to asking ``for which value of $a$ is $\frac{2\pi}{q}\Big(\frac{q}{2\pi}(\sin(2\pi/q)\log_2 x +\theta)-a\Big)$ closest to some $2n\pi \in 2\pi\mathbb{Z}$". The answer is the closest integer to $\frac{q}{2\pi}(\sin(2\pi/q)\log_2 x +\theta)$ modulo $q$ which clearly depends on $x$. Any given $a$ will therefore produce the most values of $n\leq x$ such that $\omega(n)\equiv a \mod q$ when there exists some $m\in \mathbb{Z}$ such that, $$a-\frac{1}{2} < \frac{q}{2\pi}(\sin(2\pi/q)\log_2 x +\theta)+mq < a +\frac{1}{2}.$$ A similar calculation to that in the proof of Proposition~\ref{prop7} shows that for each $a$, the set of such $x$ values has loglog density $1/q$.

We end this section with a picture of the curves
$$\cos\Big(\sin\Big(\frac{2\pi}{6}\Big)\log_2 x + \theta -\frac{2\pi a}{6}\Big)$$
for $a = 0 , 1, \ldots, 5$ plotted on a scale which makes the oscillations visible. 

\begin{figure}[H]
{
\centering
\includegraphics[scale=0.37]{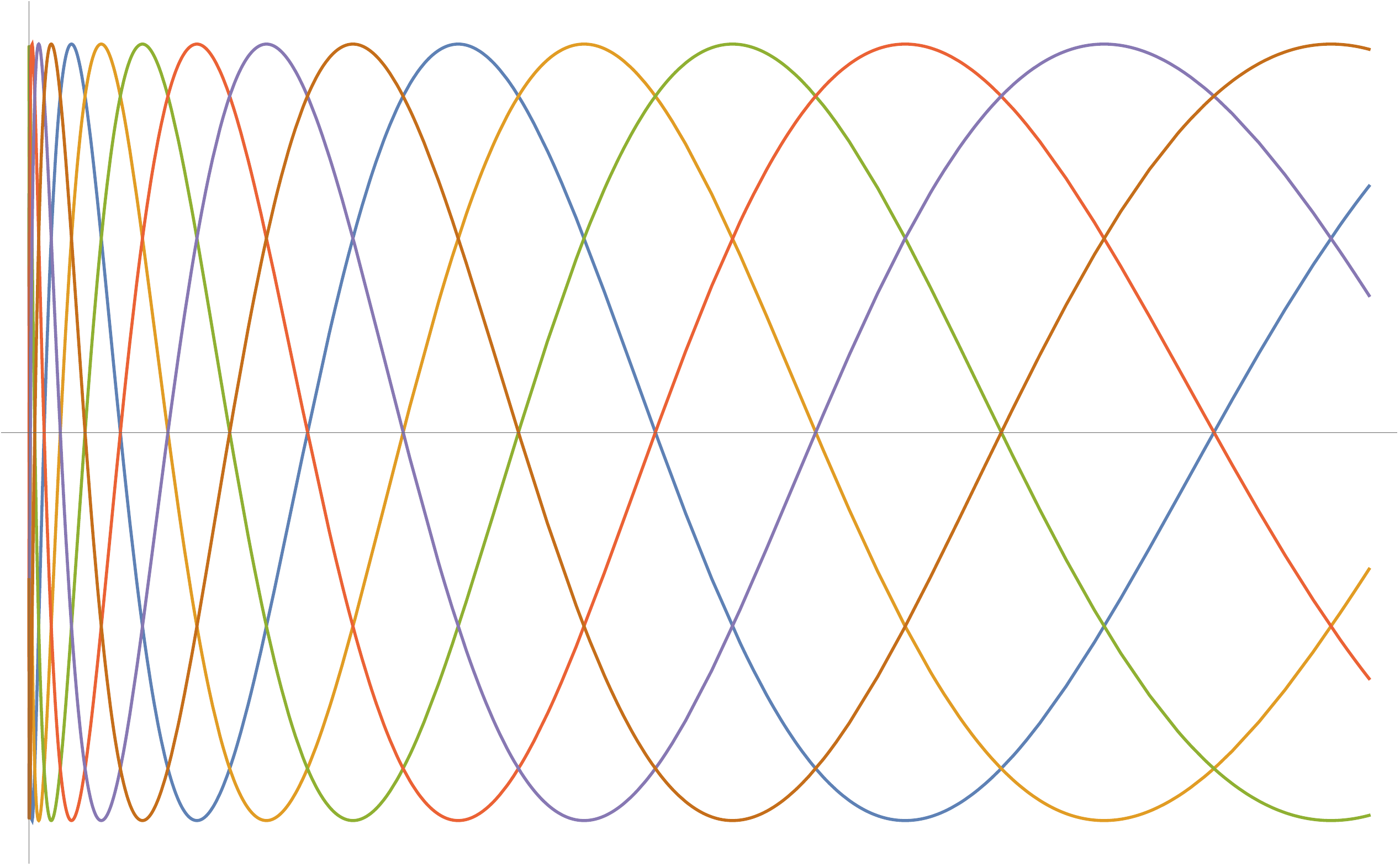}
\caption{Shifted sinusoidal curves}
}
\end{figure}

Although this picture isn't to be taken too seriously it serves as a useful illustration to have in mind for comparing the secondary terms.

\section{Proof of Theorem 3}

\begin{proof}
Let $q\geq 3$ be some fixed integer, $a \in \{0,1,\ldots, q-1\}$ and $\theta$ be defined as in Proposition~\ref{prop5}. Let $\epsilon>0$ be small and define $$U_{a,q}^{\epsilon-}:=\{x\in\mathbb{N}\::\: -\frac{\pi}{q}+\sqrt{\epsilon}<\sin(\frac{2\pi}{q})\log_2 x + \theta - \frac{2\pi a}{q} + 2n\pi < -\sqrt{\epsilon} \text{ for some }n\in \mathbb{Z}\},$$
$$ U_{a,q}^{\epsilon+}:=\{x\in\mathbb{N}\::\: -\frac{\pi}{q}-\sqrt{\epsilon}<\sin(\frac{2\pi}{q})\log_2 x + \theta - \frac{2\pi a}{q} + 2n\pi < \sqrt{\epsilon} \text{ for some }n\in \mathbb{Z}\}.$$
First let us see how, for small $\epsilon$, these sets approximate $U_{a,q}$. Our formula for $N_{a,q}(x)$ gives 
\begin{multline*}
\frac{N_{a,q}(x)-N_{b,q}(x)}{(\log x)^{\cos(\frac{2\pi}{q})-1}}=\\2|g(\phi)|\Big(\cos\Big(\sin(\frac{2\pi}{q})\log_2 x + \theta -\frac{2\pi a}{q}\Big)- \cos\Big(\sin(\frac{2\pi}{q})\log_2 x + \theta -\frac{2\pi b}{q}\Big)\Big) +o(1)
\end{multline*}
as $x \rightarrow \infty$ and for $q\geq3$ we have $g(\phi)\neq0$. Therefore, for all $\epsilon>0$ there exists some $X_0(\epsilon)$ such that for $x\geq X_0$ and for each $a, b \in \{0, \ldots, q-1\}$, $$\cos\Big(\sin(\frac{2\pi}{q})\log_2 x + \theta -\frac{2\pi a}{q}\Big)- \cos\Big(\sin(\frac{2\pi}{q})\log_2 x + \theta -\frac{2\pi b}{q}\Big)>\epsilon$$
which implies that $N_{a,q}(x)>N_{b,q}(x).$ We will use this fact to prove that for $x\geq X_0$ and $\epsilon$ sufficiently small we have
\begin{equation}
x\in U^{\epsilon-}_{a,q} \Rightarrow x \in U_{a,q}
\end{equation}
and
\begin{equation}
x \in U_{a,q} \Rightarrow x \in U^{\epsilon+}_{a,q}.
\end{equation}
It follows that $$ \sum_{\substack{ x\leq X \\ x\in U_{a, q}^{\epsilon-}}}\frac{1}{x \log x} + O_\epsilon(1) \leq \sum_{\substack{ x\leq X \\ n\in U_{a, q}}}\frac{1}{x \log x} \leq \sum_{\substack{ x\leq X \\ x\in U_{a, q}^{\epsilon+}}}\frac{1}{x \log x}+O_\epsilon(1).$$
After showing that each of $U_{a,q}^{\epsilon\pm}$ has loglog density $\frac{1}{2q}+o_\epsilon(1)$ for arbitrarily small $\epsilon$, where $o_\epsilon(1)$ is a quantity that tends to 0 as $\epsilon \rightarrow 0$, we will have shown that $U_{a,q}$ has loglog density $\frac{1}{2q}$. The result for $V_{a,q}$ is proved in much the same way.\\

\par

\textit{Proof of (3) and (4)}.

Suppose $x\geq X_0$ and $x\in U^{\epsilon-}_{a,q}$ and $\epsilon$ is small enough so that $\sin(\frac{\pi}{q})\sin(\sqrt{\epsilon})>\epsilon.$ For example, $\epsilon < 1/q^2$ will do. In order to show that $x \in U_{a,q}(x)$ we need to show
\begin{enumerate}[label=(\alph*)]
\item        $N_{a-i,q}(x)>N_{a+i,q}(x)$ for all $i \in \{ 1, 2, \ldots, [\frac{q-1}{2}] \}$ 
\item        $N_{a+i,q}(x)>N_{a-i-1,q}(x)$ for all $ i \in \{ 0, 1, \ldots, [\frac{q-2}{2}] \}$.
\end{enumerate}
To do so we will use the identity
\begin{equation}\cos(\xi + A) - \cos(\xi + B) = -2\sin\left(\frac{A-B}{2}\right)\sin\left(\xi + \frac{A+B}{2}\right).\end{equation}
For (a), let $i \in \{ 1, 2, \ldots, \text{ or } [\frac{q-1}{2}] \}$, then
\begin{multline*}
\cos\Big(\sin(\frac{2\pi}{q})\log_2 x + \theta -\frac{2\pi (a-i)}{q}\Big)- \cos\Big(\sin(\frac{2\pi}{q})\log_2 x + \theta -\frac{2\pi (a+i)}{q}\Big)\\=-2\sin(\frac{2\pi i}{q})\sin(\xi)
\end{multline*}
where $\xi = \sin(\frac{2\pi}{q})\log_2 x + \theta -\frac{2\pi a}{q} \in \left(-\frac{\pi}{q}+ \sqrt{\epsilon}, -\sqrt{\epsilon} \right)$ which is $>\epsilon.$ This proves (a). For (b), let  $ i \in \{ 0, 1, \ldots, \text{ or } [\frac{q-2}{2}] \}$, then
\begin{multline*}\cos\Big(\sin(\frac{2\pi}{q})\log_2 x + \theta -\frac{2\pi (a+i)}{q}\Big)- \cos\Big(\sin(\frac{2\pi}{q})\log_2 x + \theta -\frac{2\pi (a-i-1)}{q}\Big)\\
=2\sin(\frac{\pi(2i+1)}{q})\sin(\xi+\frac{\pi}{q})
\end{multline*}
where again $\xi \in \left(-\frac{\pi}{q}+ \sqrt{\epsilon}, -\sqrt{\epsilon} \right)$ which is again $>\epsilon.$ This proves (b) and that $x \in U_{a,q}(x).$

Now suppose $x\geq X_0$ and $x \in U_{a,q}$. To show that $x\in U_{a,q}^{\epsilon+}$ we need to find some $n\in \mathbb{Z}$ such that $$-\frac{\pi}{q}-\sqrt{\epsilon}<\sin(\frac{2\pi}{q})\log_2 x + \theta - \frac{2\pi a}{q} + 2n\pi < \sqrt{\epsilon}.$$ Suppose this is not the case. Then either we can find some $n$ such that $$-\pi-\sqrt{\epsilon} \leq \sin(\frac{2\pi}{q})\log_2 x + \theta - \frac{2\pi a}{q} + 2n\pi \leq  -\frac{\pi}{q}-\sqrt{\epsilon}.$$ In which case
\begin{multline*}
\cos\Big(\sin(\frac{2\pi}{q})\log_2 x+\theta -\frac{2\pi (a-1)}{q}\Big)-\cos\Big(\sin(\frac{2\pi}{q})\log_2 x+\theta -\frac{2\pi a}{q}\Big) \\
=-2 \sin(\frac{\pi}{q})\sin(\xi + \frac{\pi }{q})
\end{multline*} for some $ \xi\in\Big[-\pi-\sqrt{\epsilon}, -\frac{\pi}{q}-\sqrt{\epsilon}\Big]$. This is then $>\epsilon$ and so $N_{a-1, q}(x)>N_{a, q}(x)$, contrary to our assumption on $x$.
Or else we can find some $n$ such that $$\sqrt{\epsilon} \leq \sin(\frac{2\pi}{q})\log_2 x + \theta - \frac{2\pi a}{q} + 2n\pi  \leq \pi-\sqrt{\epsilon}.$$ In which case 
\begin{multline*}
\cos\Big(\sin(\frac{2\pi}{q})\log_2 x+\theta -\frac{2\pi (a+1)}{q}\Big)-\cos\Big(\sin(\frac{2\pi}{q})\log_2 x+\theta -\frac{2\pi (a-1)}{q}\Big) \\
=-2\sin( -\frac{2\pi}{q} ) \sin(\xi   )
\end{multline*} for some $ \xi\in (\sqrt{\epsilon}, \pi-\sqrt{\epsilon})$. Again making the difference $>\epsilon$ and so $N_{a+1, q}(x)>N_{a-1,q}(x)$, again contrary to our assumption on $x$. We must therefore have $x \in U_{a,q}^{\epsilon+}$.\\

It remains to calculate the loglog densities of $U^{\epsilon-}_{a,q}$ and $U^{\epsilon+}_{a,q}$. This is similar to the proof of Proposition~\ref{prop7}. For $N\in\mathbb{N}$ define $x_N:=e_2((2N\pi+2a\pi/q-\theta-\pi/q+\sqrt{\epsilon})/\sin(2\pi/q))$ and $y_N:=e_2((2N\pi+2a\pi/q-\theta-\sqrt{\epsilon})/\sin(2\pi/q))$, then for $x\geq X_0$ we have $x \in U^{\epsilon-}_{a, q}$ if and only if $x\in (x_N, y_N)$ for some $N$ and we therefore have
\begin{align*}
\liminf_{X\rightarrow \infty}\frac{1}{\log_2 X} \sum_{ \substack{n\leq x \\ n\in U^{\epsilon-}_{a, q} }}\frac{1}{n \log n} &= \lim_{N \rightarrow \infty}\frac{1}{\log_2 x_N}\sum_{m=1}^{N-1}\sum_{n\in (x_m, y_m)}\frac{1}{n\log n} \\
&=  \lim_{N \rightarrow \infty}\frac{1}{\log_2 x_N}\sum_{m=1}^{N-1}\left(\frac{(\pi/q)-2\sqrt{\epsilon}}{\sin(2\pi/q)}+O\Big(\frac{1}{m \log m}\Big)\right) \\
&= \lim_{N \rightarrow \infty}\frac{\sin(2\pi/q)}{2N\pi+2a\pi/q-\theta-\pi/q+\sqrt{\epsilon}}\Big((N-1)\frac{(\pi/q)-2\sqrt{\epsilon}}{\sin(2\pi/q)}+O(\log N) \Big)\\
&=\frac{1}{2q} - \frac{\sqrt{\epsilon}}{\pi}.
\end{align*}
Also,
\begin{align*}
\limsup_{X\rightarrow \infty}\frac{1}{\log_2 X} \sum_{ \substack{n\leq x \\ n\in U^{\epsilon-}_{a, q} }}\frac{1}{n\log n} &= \lim_{N \rightarrow \infty}\frac{1}{\log_2(y_N))}\sum_{m=1}^{N}\sum_{n\in (x_m, y_m)}\frac{1}{n\log n}=\frac{1}{2q} - \frac{\sqrt{\epsilon}}{\pi}
\end{align*}

Hence the loglog density of $U^{\epsilon-}_{a,q}$ exists and is equal to $\frac{1}{2q}-\frac{\sqrt{\epsilon}}{\pi}$. A very similar calculation shows that the loglog density of $U^{\epsilon+}_{a,q}$ is $\frac{1}{2q}+\frac{\sqrt{\epsilon}}{\pi}$ and so by (4) and the fact that $\epsilon$ can be taken arbitrarily small we can conclude that $U_{a,q}$ has loglog density $\frac{1}{2q}$.

\end{proof}

\section{Proof of Theorem 4}

As in the previous proof, let $\epsilon$ be small enough and $X_0$ large enough so that $\sin(\frac{\pi}{q})\sin(\sqrt{\epsilon})>\epsilon$ and that for $x\geq X_0$ we have $$\cos\Big(\sin(\frac{2\pi}{q})\log_2 x + \theta -\frac{2\pi a}{q}\Big)- \cos\Big(\sin(\frac{2\pi}{q})\log_2 x + \theta -\frac{2\pi b}{q}\Big)>\epsilon \Rightarrow N_{a,q}(x)>N_{b,q}(x)$$
and
$$x\in U^{\epsilon-}_{a,q} \Rightarrow x \in U_{a,q}$$
and
$$x \in U_{a,q} \Rightarrow  x \in U^{\epsilon+}_{a,q}.$$

We will show that only the permutations stated in Theorem~\ref{thm4} can occur for $x\geq X_0$. Suppose that we have some $x\geq X_0$ for which $N_{a,q}(x)$ is leading, that is, $\max_{c \:mod\: q}N_{c,q}(x) = N_{a,q}(q)$. It follows that $x\in U_{a,q}^{\epsilon+} \cup V_{a,q}^{\epsilon+}$ since otherwise there would be some integer $n$ such that
$$\frac{\pi}{q}+\sqrt{\epsilon} \leq \sin(\frac{2\pi}{q})\log_2 x + \theta - \frac{2\pi a}{q} + 2n\pi  \leq 2\pi-\frac{\pi}{q}-\sqrt{\epsilon}$$
and hence
\begin{multline*}
\cos\Big(\sin(\frac{2\pi}{q})\log_2 x+\theta -\frac{2\pi b}{q}\Big)-\cos\Big(\sin(\frac{2\pi}{q})\log_2 x+\theta -\frac{2\pi a}{q}\Big) \\
=-2\sin( \frac{\pi(a-b)}{q} ) \sin(\xi +\frac{\pi(a-b)}{q}  )
\end{multline*}
for some $\xi \in (\frac{\pi}{q}+\sqrt{\epsilon}\:,\:\: 2\pi -\frac{\pi}{q} - \sqrt{\epsilon}).$ But then this is $>\epsilon$ for either $b=a+1 \mod q$ or $b=a-1 \mod q$ contradicting the assumption that $N_{a,q}(x)$ was leading.

To prove Theorem~\ref{thm4} it suffices to prove that $\max _{\pm 1}N_{a \pm 1, q}(x) > N_{b,q}(x)$ for all $b \neq a, a\pm1 \mod q$. This follows, in a by now familiar fashion, from the fact that there exists an integer $n$ such that
$$-\frac{\pi}{q}-\sqrt{\epsilon} < \sin(\frac{2\pi}{q})\log_2 x + \theta - \frac{2\pi a}{q} + 2n\pi  < \frac{\pi}{q}+\sqrt{\epsilon}$$
since then
\begin{align*}
&\max _{\pm} \:\:\:\cos\Big(\sin(\frac{2\pi}{q})\log_2 x+\theta - \frac{2\pi(a \pm 1)}{q}\Big)-\cos\Big(\sin(\frac{2\pi}{q})\log_2 x+\theta -\frac{2\pi b}{q}\Big) \\
&=\max_{\pm} \:\:\:-2\sin( \pi\frac{b-(a \pm 1)}{q}) \sin(\xi -\frac{\pi(b-a\pm 1)}{q}  )\\
&=\max_{\pm} \:\:\: 2\sin( \pi\frac{b-(a \pm 1)}{q}) \sin(\pi\frac{(b-a\pm 1)}{q} -\xi )
\end{align*}
where $\xi \in (-\frac{\pi}{q}-\sqrt{\epsilon}\:,\:\: \frac{\pi}{q} + \sqrt{\epsilon})$, so this is $>\epsilon$ for $b \neq a, a\pm 1 \mod q$ which proves the claim.

\section*{Acknowledgements}

The author would like to thank his supervisor, Andrew Granville, for suggesting this question. This work was supported by the Engineering and Physical Sciences Research Council EP/L015234/1 via the EPSRC Centre for Doctoral Training in Geometry and Number Theory (The London School of Geometry and Number Theory), University College London.

\nocite{Mont-Vau}
\nocite{Sel-Sat}

\end{document}